\documentclass{article}

\usepackage{type1cm}       
\usepackage{makeidx}         % allows index generation
\usepackage{graphicx}        % standard LaTeX graphics tool
\usepackage{multicol}        % used for the two-column index
\usepackage[bottom]{footmisc}% places footnotes at page bottom

\usepackage[a4paper, total={6.5in, 6.5in}]{geometry}

\usepackage{bm}
\usepackage{todonotes}
\usepackage{amsmath}
\usepackage{amsfonts}
\usepackage{amssymb}
\usepackage{amsthm}
\usepackage{array}
\usepackage{float}
\usepackage{subfig}
\usepackage[ruled, lined, longend, linesnumbered]{algorithm2e}

\usepackage{pgfplots}
\pgfplotsset{compat=newest}
\usepackage{tikz}

\newtheorem{remark}{Remark}
\newtheorem{theorem}{Theorem}
\newtheorem{definition}{Definition}
\newtheorem{lemma}{Lemma}

\usepackage{newtxtext}       % 
\usepackage{newtxmath}       % selects Times Roman as basic font

\title{The fixed-stress splitting scheme for Biot's equations as a modified Richardson iteration: Implications for optimal convergence}

\author{Erlend Storvik, Jakub Wiktor Both, Jan Martin Nordbotten and Florin Adrian Radu\footnote{University of Bergen, 5007 Bergen, Norway; {\{erlend.storvik, jakub.both, jan.nordbotten, florin.radu\}@uib.no}}
} 
\date{}

\makeindex             % used for the subject index
\begin{document}

% \institute{Erlend Storvik, Jakub W.\ Both, Jan M.\ Nordbotten, Florin A.\ Radu 
% \at University of Bergen, 5007 Bergen, Norway, \email{\{erlend.storvik, jakub.both, jan.nordbotten, florin.radu\}@uib.no}}

\maketitle

\abstract{The fixed-stress splitting scheme is a popular method for iteratively solving the Biot equations. The method successively solves the flow and mechanic 
subproblems while adding a stabilizing term to the flow equation, which includes a parameter that can be chosen freely. However, the convergence properties of 
the scheme depend significantly on this parameter and choosing it carelessly might lead to a very slow, or even diverging, method. In this paper, we present a way 
to exploit the matrix structure arizing from discretizing the equations in the regime of 
impermeable porous media in order to obtain {\it a priori} knowledge of the optimal choice of this tuning/stabilization parameter.}

\section{Introduction}
\label{sec:1}
Due to many applications of societal consequence, ranging from life sciences to environmental engineering, simulation of flow in deformable porous 
media is of great interest. A choice of a model is the quasi-static Biot equations, which couples balance of linear momentum, and volume balance. In the most basic model, allowing only small deformations and fully saturated media, the equations become 
linear. However, the coupled problem has a complex structure, and it is not trivial to solve the full problem monolithically. On the other hand, there are many efficient solvers available for both porous media flow and elasticity. Therefore, splitting 
solvers are a popular alternative, where one, often using readily available software, solves the subproblems iteratively. 

In order to ensure that such splitting solvers converge, a stabilizing term is added to one, or both, of the equations. 
Choosing this term is important for the convergence properties of the scheme. Particularly, in problems with high coupling strength, 
the number of iterations to achieve convergence varies significantly for different stabilizations. In the fixed-stress splitting scheme, 
the original idea was to choose the stabilization term in order to preserve the volumetric stress over the iterations, see \cite{settari1998}, 
by adding a scaled increment of pressures to the flow equation. Convergence was proved mathematically 
in \cite{andro} and later, using a different approach, in \cite{jakubAML}. In \cite{storvik}, a range in which the optimal stabilization term recides was provided theoretically, and verified numerically. Additionally, a numerical scheme to find the parameter was proposed.

In this work, we continue the discussion on the optimal choice of the stabilization parameter of the fixed-stress splitting scheme. Particularly, for the case of impermeable porous 
media, where the coupling strength is known to be high \cite{kim}. Moreover, we expect it to be relevant for 
the more general low-permeable case. To do so, we examine the matrix 
structure of the linear problem that arises when we apply the fixed-stress splitting scheme for an idealistically impermeable problem and realize this as a modified 
Richardson iteration. Using theory for the Richardson iteration, we find the optimal stabilization parameter
and discuss how to compute it. To summarize, the contributions are: 
\begin{itemize}
 \item A proof that the fixed-stress splitting scheme can be posed as a modified Richardson iteration, with a link between the optimal constant for the modified Richardson iteration and the stabilization parameter in the fixed-stress splitting scheme. 
 \item A discussion on how to compute this stabilization parameter.
\end{itemize}

It is also worth noticing that the fixed-stress splitting scheme can be derived using a generalized gradient flow approach, \cite{jakubGradientFlow}, and can be combined with a wide range of discretizations, including space-time finite elements \cite{bause}. Moreover, it can be seen as a smoother for multigrid methods \cite{pacomultigrid}. In \cite{white}, the authors derived a relation between the fixed-stress splitting scheme and the modified Richardson iteration, and applied it as a preconditioner for Krylov subspace methods for solving the monolithic problem.

\section{The quasi-static linear Biot equations}
\label{sec:2}
The quasi-static linear Biot equations are a coupling of linear momentum balance and mass balance (see \cite{coussy}):
\begin{align}%\label{Biot}
-\nabla \cdot \bm \sigma &= \bm{f}, \label{eq:mechanics} \\ 
\frac{\partial_t m}{\rho} + \nabla \cdot \bm q &= S_f, \label{eq:flow}
\end{align}
where $\bm \sigma$ is the stress tensor of the medium, $m$ is the fluid mass, $\rho$ is the fluid density, $\bm q$ is the Darcy flux, and $\bm f$ and $S_f$ are 
the body forces and sources/sinks, respectively. Now, the St.\ Venant-Kirchhoff stress tensor for the effective stress and Darcy's law is applied:
\begin{align}
\bm \sigma &= 2 \mu \varepsilon\left( \bm u\right) + \lambda \nabla\cdot {\bm u I}-\alpha p I \label{eq:stress}, 
\\\bm q  &= -\kappa\left(\nabla p - {\bm g}\rho\right)\label{eq:darcy},
\end{align}
where ${\bm u}$ is the displacement, $\varepsilon\left({\bm u}\right) = \frac{1}{2}\left(\nabla{\bm u + \nabla \bm u}^\top\right)$ is the (linear) strain tensor, 
$\mu, \lambda$ are the Lam\'{e} parameters, $\alpha$ is the Biot-Willis constant, $p$ is the fluid pressure, respectively, ${\bm g}$ is the 
gravitational vector and $\kappa$ is the permeability. The volumetric change is asserted to be proportional to the change in pore pressure and mechanical 
displacement; $\frac{\partial_t m}{\rho} = \frac{\partial}{\partial t}\left(\frac{p}{M} + \alpha \nabla \cdot {\bm u}\right)$,  where $M$ is a compressibility 
constant. Finally, we define initial conditions $\bm u(t=0) =\bm u_0 $ and $p(t=0) = p_0$, and boundary conditions $\bm u|_{\Gamma_{N,\bm u}} = \bm u_D$, $\bm \sigma|_{\Gamma_{N,\bm u}} = \bm \sigma_N$, $p|_{\Gamma_{D,p}} = p_D$ and $\bm q|_{\Gamma_{N,p}} = \bm q_N$, where $\partial\Omega = \Gamma_{D,\bm u} \cup \Gamma_{N,\bm u} = \Gamma_{D,p}\cup \Gamma_{N,p}$, for a  Lipschitz domain, $\Omega\subset \mathbb{R}^d$, $d$ being the spatial dimension.

\section{The discretized Biot equations in matrix form}
\label{sec:3}
Discretizing the Biot equations by e.g., conforming finite elements and implicit Euler in a two-field formulation $\left(\bm u,p\right)$ (making the substitutions \eqref{eq:stress} 
and \eqref{eq:darcy} in \eqref{eq:mechanics} and \eqref{eq:flow}), the resulting linear system in each time step can be written as follows

\begin{equation}
\label{eq:matrixBiot}
\begin{pmatrix}
{\bf A} & -{\bf B}^\top \\
{\bf B} & {\bf C}
\end{pmatrix}
\begin{pmatrix}
{\bf u}_h \\ {\bf p}_h
\end{pmatrix}
=
\begin{pmatrix}
\bf f \\ \bf g
\end{pmatrix},
\end{equation}
where ${\bf A}$ is the linear elasticity matrix, ${\bf B}$ is the coupling matrix, ${\bf C}$ is the single-phase flow matrix, 
${\bf f}$ and ${\bf g}$ correspond to the body forces and sources/sinks, respectively, and ${\bf u}_h$ and ${\bf p}_h$ are 
the coefficient vectors for the discretized displacement and pressure, respectively. For the rest of this paper, we consider an inf-sup stable finite element pair $({\bf V}_h,Q_h)$. Furthermore, for impermeable porous media, the submatrix 
${\bf C}$ reduces to $\frac{1}{M}{\bf M}$, where ${\bf M}$ is the mass matrix. 
%The discretized Biot equations attain a classical saddle point structure,
% $\bm u_h=\sum_i [{\bf u}_h]_i\bm\varphi^{\bm V_h}_i$, $p_h=\sum_i [{\bf p}_h]_i\varphi^{Q_h}_i$,
% \begin{equation}
% \label{eq:matrixBiotLimit}
% \begin{pmatrix}
% {\bf A} & {\bf -B}^\top \\
% {\bf B} & {\bf 0}
% \end{pmatrix}
% \begin{pmatrix}
% {\bf u}_h \\ {\bf p}_h
% \end{pmatrix}
% =
% \begin{pmatrix}
% \bf f \\ \bf g
% \end{pmatrix},
% \end{equation}
% or equivalently,
% \begin{equation}
% \label{eq:matrixBiotReduced}
% \begin{pmatrix}
% {\bf A} & {\bf -B}^\top \\
% {\bf 0}& {\bf S}
% \end{pmatrix}
% \begin{pmatrix}
% {\bf u}_h \\ {\bf p}_h
% \end{pmatrix}
% =
% \begin{pmatrix}
% \bf f \\ \bf \tilde{g}
% \end{pmatrix}
% \end{equation}
% where ${\bf S}={\bf B}{\bf A}^{-1}{\bf B}^\top$ is the Schur complement and ${\bf\tilde{g}}={\bf g}-{\bf BA}^{-1}{\bf f}$.

\section{The fixed-stress splitting scheme as a modified Richardson iteration}
As mentioned in the introduction, the fixed-stress splitting scheme adds a scaled incremental pressure to the flow equation while eliminating its dependence on the 
displacement. For impermeable media, $\kappa=0$, this results in the following linear system
\begin{equation}
\label{eq:matrixFS}
\begin{pmatrix}
{\bf A} & -{\bf B}^\top \\
{\bf 0}& \left(L+\frac{1}{M}\right){\bf M}
\end{pmatrix}
\begin{pmatrix}
\Delta{\bf u}_h^i \\ \Delta{\bf p}_h^i
\end{pmatrix}
=
\begin{pmatrix}
\bf f \\ \bf g
\end{pmatrix}
-
\begin{pmatrix}
{\bf A} & {\bf -B}^\top \\
{\bf B} & {\frac{1}{M}\bf M}
\end{pmatrix}
\begin{pmatrix}
{\bf u}_h^{i-1}\\ {\bf p}_h^{i-1}
\end{pmatrix},
\end{equation}
where  $L$ is the stabilization parameter, $i\geq1$ is the iteration index, $\Delta {\bf u}_h^{i}={\bf u}_h^{i}-{\bf u}_h^{i-1}$ and $\Delta {\bf p}_h^{i}={\bf p}_h^{i}-{\bf p}_h^{i-1}$. 
% In the limit case of incompressible fluids and impermeable porous media we look for a solution to the system
% \begin{equation}
% \label{eq:matrixFSlimit}
% \begin{pmatrix}
% {\bf A} & {\bf -B}^\top \\
% {\bf 0}& L{\bf M}
% \end{pmatrix}
% \begin{pmatrix}
% \Delta{\bf u}_h^i \\ \Delta{\bf p}_h^i
% \end{pmatrix}
% =
% \begin{pmatrix}
% \bf f \\ \bf g
% \end{pmatrix}
% -
% \begin{pmatrix}
% {\bf A} & {\bf -B}^\top \\
% {\bf B} & {\bf 0}
% \end{pmatrix}
% \begin{pmatrix}
% {\bf u}_h^{i-1}\\ {\bf p}_h^{i-1}
% \end{pmatrix}.
% \end{equation}

\begin{remark}[Optimality of alternative splitting]
 Notice that for impermable media, $\kappa = 0$, and particular discretizations, e.g., (${\bf P}_1,P_0$), the undrained split \cite{kimundrained}, will converge in one iteration. Nevertheless, our experience is that the optimized fixed-stress splitting proposed here is superior for slight perturbations of the permeability. 
\end{remark}

In the original formulation of the fixed-stress splitting scheme \cite{settari1998} the constant $L=\tfrac{\alpha^2}{K_{dr}}$ was chosen, where 
$K_{dr}=\tfrac{2\mu}{d}+\lambda$, is the physical drained bulk modulus. Later, in \cite{andro}, convergence was proved for $L\geq\tfrac{\alpha^2}{2K_{dr}}$.
In \cite{storvik}, an interval containing the optimal stabilization parameter, including both of the two aforementioned parameters, was provided through mathematical proofs, and then verified numerically. We now show that the optimal 
parameter for impermeable media is a value in this domain that is possible to compute {\it a priori}. For this, we need the mathematical bulk modulus.
\begin{definition}
The mathematical bulk modulus, $K^\star_{dr}\geq K_{dr}$, is defined as the largest constant such that 
\begin{equation}\label{Kdr}
2 \mu \left\|\varepsilon\left(\bm u_h\right)\right\|^2+\lambda \|\nabla \cdot{\bm u_h} \|^2 \geq K^\star_{dr}\|\nabla \cdot{\bm u_h} \|^2\qquad\text{for all }\bm{u_h} \in {\mathbf V}_h.
\end{equation}
\end{definition}
It is easily seen that the physical drained bulk modulus satisfies inequality \eqref{Kdr}, but generally the bound is not sharp. 

Furthermore, by assuming that we have an inf-sup stable discretization we are able to define the following parameter, $\beta$, 
that is important in finding the optimal $L$. 
\begin{lemma}\label{Lemma1} Assume that the pair $({\bf V}_h, Q_h)$ is inf-sup stable. 
There exists $\beta >0$ such that for any $p_h \in Q_h$ there exists a ${\bm u}_h \in {\bf V}\!_h $ satisfying
$\left\langle \nabla \cdot {\bm u_h} ,q_h \right\rangle = \left\langle p_h, q_h \right\rangle $ for all $q_h \in Q_h$ and 
\begin{equation}
2 \mu\left\|{\bm \varepsilon(\bm u_h })\right\|^2 + \lambda \left\| \nabla\cdot {\bm u_h}\right\|^2\leq \beta \left\|p_h\right\|^2. \label{beta}
\end{equation}
\end{lemma}
A proof of this lemma can be found in \cite{storvik}.
\begin{theorem}\label{theorem:fs-richardson}
For impermable media, $\kappa = 0$, the fixed-stress splitting scheme \eqref{eq:matrixFS} can be interpreted as the modified Richardson iteration 
\begin{equation}\label{eq:FSRichardson}
{\bf p}_h^{i}={\bf p}_h^{i-1}+\omega\left({\bf M}^{-1}{\bf \tilde{g}} - {\bf M}^{-1}{\bf Sp}_h^{i-1}\right)
\end{equation} where $\omega=\left(L+\frac{1}{M}\right)^{-1}$, ${\bf \tilde{g} = g-A}^{-1}{\bf B f}$, and ${\bf S}=\frac{1}{M}{\bf M}+{\bf B}{\bf A}^{-1}{\bf B}^\top$ is the Schur complement. For the error ${\bf e}_p^i:= {\bf p}_h^i - {\bf p}_h$ it holds for all $i\geq1$
\begin{equation}\label{eq:FSRichardson-rate}
 \left\langle {\bf M} {\bf e}_p^i, {\bf e}_p^i\right\rangle 
 \leq 
 \left\| {\bf I} - \omega {\bf M}^{-1/2}{\bf S}{\bf M}^{-1/2} \right\|_2^2 \,
 \left\langle {\bf M} {\bf e}_p^{i-1}, {\bf e}_p^{i-1}\right\rangle.
\end{equation}
From that, the optimal choice of $\omega$ is 
\begin{equation}\label{eq:OptimalLimitCase}
\omega_{\mathrm{opt}}=\dfrac{2}{\lambda_\mathrm{max}\left( {\bf M}^{-\frac{1}{2}}{\bf S}{\bf M}^{-\frac{1}{2}}  \right)+\lambda_\mathrm{min}\left({\bf M}^{-\frac{1}{2}}{\bf S}{\bf M}^{-\frac{1}{2}}  \right)}
\end{equation}
with the identifications $$\lambda_\mathrm{max}\left( {\bf M}^{-\frac{1}{2}}{\bf S}{\bf M}^{-\frac{1}{2}}  \right)=\frac{1}{M}+\frac{\alpha^2}{K_{dr}^\star}\quad 
\mathrm{and} \quad \lambda_\mathrm{min}\left({\bf M}^{-\frac{1}{2}}{\bf S}{\bf M}^{-\frac{1}{2}}  \right)=\frac{1}{M}+\frac{\alpha^2}{\beta},$$ where $\alpha$ is the Biot-Willis coupling constant, $M$ is a compressibility coefficient, $K_{dr}^\star$ is the mathematical bulk modulus and $\beta$ is the constant from Lemma~\ref{Lemma1}. 
\end{theorem}
\begin{proof}
From \eqref{eq:matrixFS} we have that 
\begin{equation}\label{eq:2}
\Delta {\bf p}_h^{i}=\left(\left(L+\frac{1}{M}\right){\bf M}\right)^{-1}\left({\bf g} - {\bf B}{\bf u}_h^{i-1}-\frac{1}{M}{\bf M}{\bf p}_h^{i-1}\right)
\end{equation}
and
\begin{equation}\label{eq:3}
{\bf u}_h^{i-1}={\bf A}^{-1}{\bf f} +{\bf  A}^{-1}{\bf B}^\top{\bf p}_h^{i-1}.
\end{equation}
From the update of pressures in the fixed-stress splitting scheme we get the modified Richardson iteration \eqref{eq:FSRichardson}

$${\bf p}_h^{i}={\bf p}_h^{i-1}+\Delta {\bf p}_h^{i}= 
{\bf p}_h^{i-1}+\left(L+\frac{1}{M}\right)^{-1}\left({\bf M}^{-1}{\bf \tilde{g}}-{\bf M}^{-1}{\bf S}{\bf p}_h^{i-1}\right).$$

To find the optimal choice of the parameter $\omega$ in \eqref{eq:FSRichardson} we modify the equation slightly by making the substitution ${\bf p}_h^{i}={\bf M}^{-\frac{1}{2}}{\bf \tilde{p}}_h^{i}$ and multiply from the left by ${\bf M}^\frac{1}{2}$ to get 
$${\bf \tilde{p}}_h^{i}={\bf \tilde{p}}_h^{i-1}+\left(L+\frac{1}{M}\right)^{-1}\left({\bf M}^{-\frac{1}{2}}{\bf \tilde{g}}-{\bf M}^{-\frac{1}{2}}{\bf S}{\bf M}^{-\frac{1}{2}}{\bf \tilde{p}}_h^{i-1}\right),$$
where ${\bf M}^{-\frac{1}{2}}{\bf S}{\bf M}^{-\frac{1}{2}}$ is symmetric. By the standard theory for the modified Richardson iteration~\cite{Saad2003}, we conclude 
the optimal tuning parameter~\eqref{eq:FSRichardson} and corresponding rate~\eqref{eq:FSRichardson-rate}. 
Now, to make the identification  $\lambda_\mathrm{max}\left( {\bf M}^{-\frac{1}{2}}{\bf S}{\bf M}^{-\frac{1}{2}}  \right)=\dfrac{\alpha^2}{K_{dr}^\star}+\dfrac{1}{M}$ we consider Rayleigh quotients,

\begin{align*}
\lambda_\mathrm{max}\left( {\bf M}^{-\frac{1}{2}}{\bf S}{\bf M}^{-\frac{1}{2}}  \right) &= \sup_{\bf p\neq 0}\dfrac{{\bf p}^\top{\bf M}^{-\frac{1}{2}}{\bf S}{\bf M}^{-\frac{1}{2}}{\bf p}}{{\bf p}^\top{\bf p}}=
\frac{1}{M} + \sup_{{\bf p\neq0, Au}={\bf Bp}}\frac{{\bf u}^\top{\bf Au}}{{\bf p}^\top{\bf Mp}}\\&=\frac{1}{M} + \sup_{0\neq p_h\in Q_h}\frac{2\mu\left\|\varepsilon\left(\bm u_h\right)\right\|^2+\lambda\left\|\nabla\cdot \bm u_h\right\|^2}{\left\|p_h\right\|^2},
\end{align*}
where $\bm u_h\in \bm V_h$ solves 
$$2\mu\left\langle \varepsilon\left(\bm u_h\right),\varepsilon\left(\bm v_h\right)\right\rangle + \lambda \left\langle \nabla \cdot \bm u_h,\nabla\cdot\bm v_h\right\rangle = \alpha\left\langle p_h, \nabla\cdot\bm v_h\right\rangle \quad \mathrm{for\ all} 
\quad \bm v_h\in \bm V_h,$$ for given $p_h\in Q_h$.
Testing with $\bm v_h=\bm u_h$ we have 
\begin{eqnarray*}
&2\mu\left\| \varepsilon\left(\bm u_h\right)\right\|^2 + \lambda \left\| \nabla \cdot \bm u_h\right\|^2 =  \alpha\left\langle p_h, \nabla\cdot\bm u_h\right\rangle \leq \alpha\left\|p_h\right\|\left\|\nabla\cdot\bm u_h\right\|\\
&\leq \frac{\alpha}{\sqrt{K_{dr}^\star}}\left\|p_h\right\|\sqrt{2\mu\left\| \varepsilon\left(\bm u_h\right)\right\|^2 + \lambda \left\| \nabla \cdot \bm u_h\right\|^2 }
\end{eqnarray*}
by the Cauchy-Schwarz inequality and the definition of ${K_{dr}^\star}$. This implies that
$$\lambda_\mathrm{max}\left( {\bf M}^{-\frac{1}{2}}{\bf S}{\bf M}^{-\frac{1}{2}}  \right)=\frac{1}{M} + \frac{\alpha^2}{K_{dr}^\star}.$$
For the identification $\lambda_\mathrm{min}\left( {\bf M}^{-\frac{1}{2}}{\bf S}{\bf M}^{-\frac{1}{2}}  \right)=\dfrac{1}{M}+\dfrac{\alpha^2}{\beta}$, recognize that 
$\lambda_\mathrm{min}\left( {\bf M}^{-\frac{1}{2}}{\bf S}{\bf M}^{-\frac{1}{2}}  \right) = 
\dfrac{1}{M}+\lambda_\mathrm{min}\left( {\bf M}^{-\frac{1}{2}}{\bf B^\top A}^{-1}{\bf B}{\bf M}^{-\frac{1}{2}}  \right)$, and consider the 
algebraic form of Lemma~\ref{Lemma1},
\begin{equation}\label{eq:lemMatrix}
\exists \beta \ \forall {\bf p}_h\ \exists {\bf u}_h\ : \ {\bf u}_h^\top{\bf A u}_h\leq \beta{\bf p}_h^\top{\bf M p}_h \ \mathrm{and} \ \frac{1}{\alpha}{\bf B}^\top{\bf u}_h={\bf M p}_h.
\end{equation}
Moreover, recognize that for inf-sup stable discretizations ${\bf B^\top A}^{-1}{\bf B}$ is invertible. Let ${\bf \tilde{u}}_h$ be the minimizer of $\left\{ {\bf u}_h^\top{\bf A u}_h \ : \ \dfrac{1}{\alpha}{\bf B}^\top{\bf u}_h={\bf M p}_h\right\}$. Finding the saddle point using the Lagrangian 
$$\mathcal{L}\left({\bf u}_h,{\bf\Lambda}\right) = {\bf u}_h^\top{\bf A u}_h+{\bf \Lambda}^\top\left( \dfrac{1}{\alpha}{\bf B}^\top{\bf u}_h-{\bf M p}_h\right)$$
we get the solutions
${\bf \Lambda}=2\alpha^2{\left({\bf B^\top A}^{-1}{\bf B}\right)}^{-1}{\bf Mp_h}$ and ${\bf \tilde{u}}_h=\dfrac{1}{2\alpha}{\bf A}^{-1}{\bf B\Lambda}=
\alpha{\bf B^{-\top}}{\bf Mp}_h$. 
Hence, the minimizer satisfies $${\bf \tilde{u}}^\top_h{\bf A \tilde{u}}_h=\alpha^2{\bf p}_h^\top{\bf M}{{\left({\bf B^\top A}^{-1}{\bf B}\right)}^{-1}}{\bf Mp}_h.$$ 
From \eqref{eq:lemMatrix} we get, with ${\bf u}_h$ depending on ${\bf p}_h$ as above, 
\begin{eqnarray*}
&\beta = \sup_{{\bf p}_h\neq {\bf 0}} \dfrac{{\bf \tilde{u}}^\top_h{\bf A \tilde{u}}_h}{{\bf p}_h^\top{\bf M p}_h}=\alpha^2\sup_{{\bf p}_h\neq {\bf 0}} \dfrac{{\bf p}_h^\top{\bf M}{{\left({\bf B^\top A}^{-1}{\bf B}\right)}^{-1}}{\bf Mp}_h}{{\bf p}_h^\top{\bf M p}_h}
\\&= \alpha^2\sup_{{\bf p}_h\neq {\bf 0}} \dfrac{{\bf p}_h^\top{\bf M}^\frac{1}{2}{{\left({\bf B^\top A}^{-1}{\bf B}\right)}^{-1}}{\bf M}^\frac{1}{2}{\bf p}_h}{{\bf p}_h^\top{\bf  p}_h}
 = \ \alpha^2\lambda_{\mathrm{min}}\left({\bf M}^{-\frac{1}{2}}{{\bf B^\top A}^{-1}{\bf B}}{\bf M}^{-\frac{1}{2}}\right)^{-1}.
\end{eqnarray*}
\end{proof}
As a result of Theorem~\ref{theorem:fs-richardson} we get the optimal stabilization parameter $L_\mathrm{opt}=\frac{\alpha^2}{2}\left(\frac{1}{K_{dr}^{\star}}+\frac{1}{\beta}\right)$, which includes information on the boundary conditions.

\subsection{Computing the optimal stabilization parameter}
There are two main options now for finding the optimal stabilization parameter. One could estimate $K_{dr}^\star$ by $K_{dr}$ and find $\beta$ by some table search of inf-sup constants, but our experience is that this generally is not good enough. Therefore, we suggest to approximate the eigenvalues $\lambda_\mathrm{max}\left( {\bf M}^{-\frac{1}{2}}{\bf S}{\bf M}^{-\frac{1}{2}}  \right)$ and $\lambda_\mathrm{min}\left({\bf M}^{-\frac{1}{2}}{\bf S}{\bf M}^{-\frac{1}{2}}  \right)$ using the power iteration \cite{Saad2003} or some other cheap inexact scheme for finding maximal and minimal eigenvalues. Notice that this is possible to do without explicitly computing ${\bf S}$. Also, the computation of $K_{dr}^\star$ is relatively cheap while the computation of $\beta$ is not -- it involves applying the fixed-stress splitting scheme with a non-optimal stabilization parameter. We suggest using coarse approximations of both to define an approximation of $L_\mathrm{opt}$. This approximation can be expected to still yield relatively good performance of the fixed-stress splitting scheme.

\section{Numerical examples}

We test Theorem~\ref{theorem:fs-richardson} numerically including the optimality of the proposed stabilization parameter. For this, we consider a unit square test case with source terms that 
enforce parabolic displacement and pressure profiles and zero Dirichlet boundary conditions everywhere, but at the top boundary for the momentum balance equation where zero Neumann conditions are considered, 
see Figure~\ref{fig:limit-performance}(b) or \cite{storvik} for a more thorough explanation. We choose the parameters from Table~\ref{tab:usna}. The choice of the extreme values $M=\infty$ and 
$\kappa=0$ mimics the limit scenario of incompressibility and impermeability.
In this regime, the coupling strength of the Biot equations as defined by~\cite{kim} is infinite, and we consider the test a suitable stress test.

For fixed physical parameters, we test the performance of the fixed-stress splitting scheme employing different tuning parameters. For the inf-sup stable P2-P1 discretization, the average number of iterations is presented in Figure~\ref{fig:limit-performance}(a). Notice especially that the proposed optimal stabilization parameter truly is optimal.

\begin{table}[h]
\centering
\begin{tabular}{ l | c | r }
Name & Symbol & Value\\
\hline
Lam\'e parameters & $\mu$, $\lambda$ & $41.667\cdot10^{9}$, $27.778\cdot10^{9}$\\

Permeability and compressibility & $\kappa$, $\frac{1}{M}$& $0$, $0$ \\

Temporal parameters& $t_0$, $\tau$, $T$ & $0$, $0.1$, $1$ \\

Biot-Willis coefficient& $\alpha$ & $1$ \\

Relative error tolerance& $\epsilon_r$ & $10^{-6}$  \\

Inverse of mesh size diameter & $1/h$ & $16, 32, 64, 128$\\
 
\end{tabular}
\caption{Table of coefficients}
\label{tab:usna}
\end{table}

\begin{figure}[h!]
\begin{minipage}{0.5\textwidth}
\centering
\subfloat[Numerical results]{\begin{tikzpicture}
\begin{axis}[
    ylabel={Avg.\ \#iter.\ pr.\ time step},
    width = 0.9\textwidth,
    xlabel={$\mathcal{D} \cdot10^{-11}$},
   legend entries={{$1/h=16$, $1/h=32$, $1/h=64$, $1/h=128$}},
    legend cell align=left,
    legend columns=2,
    legend style={at={(0.05,0.8)},anchor=west},
   xmin = 0.65,
    xmax = 1.6,
    ymin = 10,
    ymax = 70 
]
%%%%%%%%%%%%%%%
% MARKS
%%%%%%%%%%%%%%%
\addplot [solid, mark=*, mark repeat=1, mark size =3, color=red, only marks]   coordinates {
(0.690, 41.6)
(0.766, 37.4)
(0.843, 33.9)
(0.919, 31.0)
(0.996, 28.5)
% (1.011, 28.0)
% (1.026, 27.5)
% (1.041, 27.2)
% (1.057, 26.8)
(1.072, 26.3)
% (1.087, 26.0)
% (1.103, 25.4)
% (1.118, 25.2)
% (1.133, 24.9)
(1.149, 24.4)
% (1.164, 24.2)
% (1.179, 23.7)
% (1.194, 23.4)
% (1.210, 23.3)
(1.225, 23.5)
% (1.240, 24.5)
% (1.256, 26.0)
(1.271, 27.5)
% (1.286, 29.4)
(1.302, 31.3)
(1.332, 35.9)
(1.363, 41.9)
(1.393, 49.9)
(1.424, 61.1)
(1.455, 78.4)
(1.531, 231.9)
};
\addplot [solid, mark=square*, mark repeat=1, mark size = 3, color=teal, only marks]   coordinates {
(0.690, 45.3)
(0.766, 40.7)
(0.843, 37.1)
(0.919, 33.9)
(0.996, 31.3)
% (1.011, 30.8)
% (1.026, 30.3)
% (1.041, 29.9)
% (1.057, 29.4)
(1.072, 28.9)
% (1.087, 28.5)
% (1.103, 28.0)
% (1.118, 27.6)
% (1.133, 27.3)
(1.149, 26.9)
% (1.164, 26.5)
% (1.179, 26.3)
% (1.194, 26.5)
% (1.210, 27.9)
(1.225, 29.7)
% (1.240, 31.9)
% (1.256, 34.3)
(1.271, 37.1)
% (1.286, 40.4)
(1.302, 44.2)
(1.332, 53.9)
(1.363, 68.5)
(1.393, 93.1)
};
\addplot [solid, mark=triangle*, mark repeat=1, mark size = 4, color=green, only marks]   coordinates {
(0.690, 47.6)
(0.766, 42.9)
(0.843, 39.1)
(0.919, 35.9)
(0.996, 33.0)
% (1.011, 32.5)
% (1.026, 32.0)
% (1.041, 31.6)
% (1.057, 31.0)
(1.072, 30.7)
% (1.087, 30.1)
% (1.103, 29.7)
% (1.118, 29.3)
% (1.133, 28.9)
(1.149, 28.5)
% (1.164, 28.2)
% (1.179, 28.5)
% (1.194, 30.0)
% (1.210, 32.3)
(1.225, 34.9)
% (1.240, 38.1)
% (1.256, 41.6)
(1.271, 45.8)
% (1.286, 51.0)
(1.302, 57.1)
(1.332, 74.8)
};
\addplot [solid, mark=diamond*, mark repeat=1, mark size = 4, color=blue, only marks]   coordinates {
(0.690, 49.1)
(0.766, 44.3)
(0.843, 40.3)
(0.919, 36.9)
(0.996, 34.0)
% (1.011, 33.6)
% (1.026, 33.0)
% (1.041, 32.6)
% (1.057, 32.0)
(1.072, 31.7)
% (1.087, 31.0)
% (1.103, 30.7)
% (1.118, 30.3)
% (1.133, 29.9)
(1.149, 29.4)
% (1.164, 29.4)
% (1.179, 30.4)
% (1.194, 32.6)
% (1.210, 35.5)
(1.225, 38.5)
% (1.240, 42.4)
% (1.256, 46.8)
(1.271, 52.3)
% (1.286, 59.2)
(1.302, 67.5)
(1.332, 93.8)
};
%%%%%%%%%%%%%%%
% LINES
%%%%%%%%%%%%%%%
\addplot [solid, color=red]   coordinates {
(0.690, 41.6)
(0.766, 37.4)
(0.843, 33.9)
(0.919, 31.0)
(0.996, 28.5)
(1.011, 28.0)
(1.026, 27.5)
(1.041, 27.2)
(1.057, 26.8)
(1.072, 26.3)
(1.087, 26.0)
(1.103, 25.4)
(1.118, 25.2)
(1.133, 24.9)
(1.149, 24.4)
(1.164, 24.2)
(1.179, 23.7)
(1.194, 23.4)
(1.210, 23.3)
(1.225, 23.5)
(1.240, 24.5)
(1.256, 26.0)
(1.271, 27.5)
(1.286, 29.4)
(1.302, 31.3)
(1.332, 35.9)
(1.363, 41.9)
(1.393, 49.9)
(1.424, 61.1)
(1.455, 78.4)
(1.531, 231.9)
};
\addplot [solid, color=teal]   coordinates {
(0.690, 45.3)
(0.766, 40.7)
(0.843, 37.1)
(0.919, 33.9)
(0.996, 31.3)
(1.011, 30.8)
(1.026, 30.3)
(1.041, 29.9)
(1.057, 29.4)
(1.072, 28.9)
(1.087, 28.5)
(1.103, 28.0)
(1.118, 27.6)
(1.133, 27.3)
(1.149, 26.9)
(1.164, 26.5)
(1.179, 26.3)
(1.194, 26.5)
(1.210, 27.9)
(1.225, 29.7)
(1.240, 31.9)
(1.256, 34.3)
(1.271, 37.1)
(1.286, 40.4)
(1.302, 44.2)
(1.332, 53.9)
(1.363, 68.5)
(1.393, 93.1)
};
\addplot [solid, color=green]   coordinates {
(0.690, 47.6)
(0.766, 42.9)
(0.843, 39.1)
(0.919, 35.9)
(0.996, 33.0)
(1.011, 32.5)
(1.026, 32.0)
(1.041, 31.6)
(1.057, 31.0)
(1.072, 30.7)
(1.087, 30.1)
(1.103, 29.7)
(1.118, 29.3)
(1.133, 28.9)
(1.149, 28.5)
(1.164, 28.2)
(1.179, 28.5)
(1.194, 30.0)
(1.210, 32.3)
(1.225, 34.9)
(1.240, 38.1)
(1.256, 41.6)
(1.271, 45.8)
(1.286, 51.0)
(1.302, 57.1)
(1.332, 74.8)
};
\addplot [solid, color=blue]   coordinates {
(0.690, 49.1)
(0.766, 44.3)
(0.843, 40.3)
(0.919, 36.9)
(0.996, 34.0)
(1.011, 33.6)
(1.026, 33.0)
(1.041, 32.6)
(1.057, 32.0)
(1.072, 31.7)
(1.087, 31.0)
(1.103, 30.7)
(1.118, 30.3)
(1.133, 29.9)
(1.149, 29.4)
(1.164, 29.4)
(1.179, 30.4)
(1.194, 32.6)
(1.210, 35.5)
(1.225, 38.5)
(1.240, 42.4)
(1.256, 46.8)
(1.271, 52.3)
(1.286, 59.2)
(1.302, 67.5)
(1.332, 93.8)
};
%%%%%%%%%%%%%%%
% Optimal L
%%%%%%%%%%%%%%%
\addplot [
color=red,
dashed,
line width=1.0pt]
table[row sep=crcr]{
1.2 -10\\
1.22 100\\
};
\addplot [
color=teal,
dashed,
line width=1.0pt]
table[row sep=crcr]{
1.19 -10\\
1.19 100\\
};
\addplot [
color=green,
dashed,
line width=1.0pt]
table[row sep=crcr]{
1.17 -10\\
1.17 100\\
};
\addplot [
color=blue,
dashed,
line width=1.0pt]
table[row sep=crcr]{
1.16 -10\\
1.16 100\\
};
\end{axis}
\end{tikzpicture}}
\end{minipage}
\begin{minipage}{0.48\textwidth}
\vspace{-15pt}
\subfloat[Displacement $\left(\left|\bm u_h \right|\right)$]{\centering \includegraphics[width=0.9\textwidth]{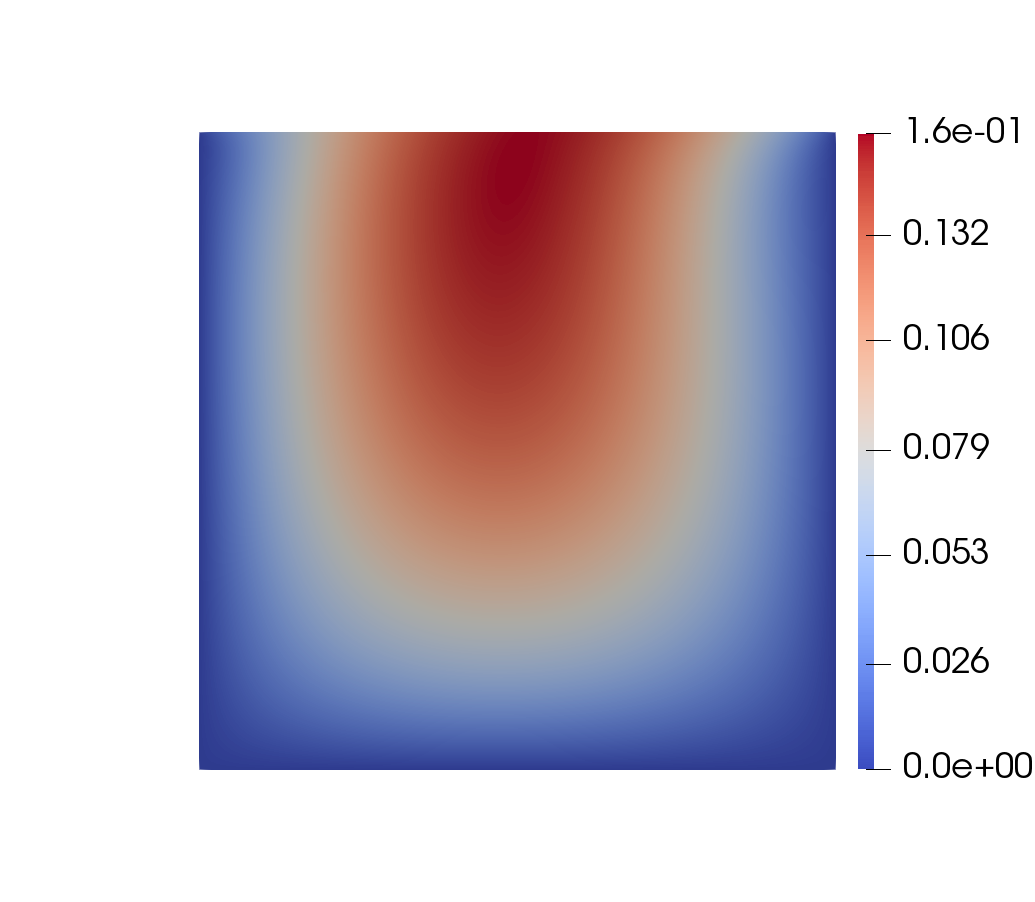}}
\end{minipage}
\caption{Average number of iterations per time step for different stabilization parameters, $L=\frac{\alpha^2}{\mathcal{D}}$, using parameters from Table~\ref{tab:usna}. 
The dashed lines highlight the location of the tuning parameter, 
$L_\mathrm{opt}$, computed using fine tolerances in the power iteration.}
\label{fig:limit-performance}
\end{figure}
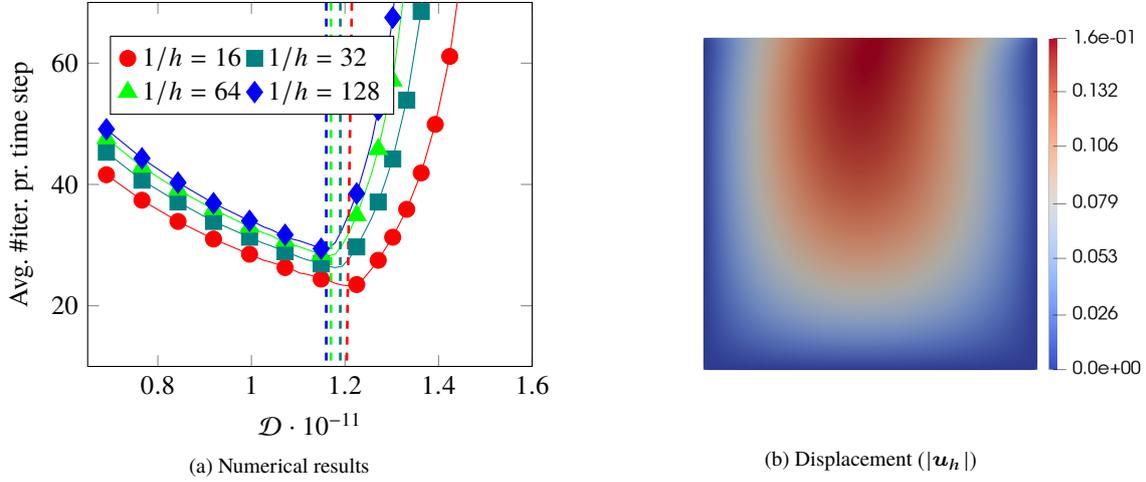

\section{Conclusions}
In this work, we derived mathematically a way to {\it a priori} determine the optimal stabilization parameter for the fixed-stress splitting scheme using the theory for the modified Richardson iteration. This optimal parameter involves the computations of maximal eigenvalues for rather large matrices. However, they do not need to be computed to high accuracy. Through a numerical experiment, we showed that the proposed stabilization parameter is optimal for this problem. 

\bibliographystyle{unsrt}
\bibliography{optimal_stabilization_parameter}
\end{document}